\documentclass[11pt]{article}
\usepackage{amsmath}
\usepackage{amssymb,amsbsy,amsthm}
\usepackage{graphicx}
\usepackage[dvipsnames]{xcolor}
\usepackage{enumerate}
\usepackage[margin=1in]{geometry}
\usepackage{hyperref}

\allowdisplaybreaks[1]

\numberwithin{equation}{section}

\newcommand{\p}{\mathbb{P}} 
\newcommand{\E}{\mathbb{E}} 
\newcommand{\eps}{\varepsilon} 
\newcommand{\TV}{\mathrm{TV}} 

\newcommand{\Poi}{\mathrm{Poi}} 

\newcommand{\wt}{\widetilde} 

\newcommand{\MAP}{\mathrm{MAP}} 

\newcommand{\Rev}{\mathrm{Rev}} 

\def\tE{\texttt{E}}

\newtheorem{theorem}{Theorem}[section]
\newtheorem{lemma}[theorem]{Lemma}

\begin{document}

\title{How fragile are information cascades?}
\author{
	Yuval Peres
	\thanks{Microsoft Research; \texttt{peres@microsoft.com}.}
	\and
	Mikl\'os Z.\ R\'acz
	\thanks{Princeton University; \texttt{mracz@princeton.edu}.}
	\and
	Allan Sly
	\thanks{Princeton University; \texttt{asly@math.princeton.edu}.}
	\and
	Izabella Stuhl
	\thanks{Penn State University; \texttt{ius68@psu.edu}.}
}
\date{\today}

\maketitle


\begin{abstract}
It is well known that sequential decision making may lead to information cascades.
That is, when agents make decisions based on their private information, as well as observing the actions of those before them,
then it might be rational to ignore their private signal and imitate the action of previous individuals.
If the individuals are choosing between a right and a wrong state,
and the initial actions are wrong, then the whole cascade will be wrong.
This issue is due to the fact that cascades can be based on very little information.

We show that if agents occasionally disregard the actions of others and base their action only on their private information,
then wrong cascades can be avoided. Moreover, we study the optimal asymptotic rate at which the error probability at time $t$ can go to zero.
The optimal policy is for the player at time $t$ to follow their private information with probability $p_{t} = c/t$,
leading to a learning rate of $c'/t$, where the constants $c$ and $c'$ are explicit.
\end{abstract}


\section{Introduction} \label{sec:intro} 

Many everyday situations involve sequential decision making, where one makes a decision based on some private information and the previous actions of others.
Consider, for example, the following classroom experiment (see~\cite[Chapter~16]{easley2010networks}). 
An experimenter puts an urn containing three balls at the front of the room. 
This urn is either \emph{majority blue}, containing two blue balls and one yellow ball, 
or \emph{majority yellow}, containing two yellow balls and one blue ball; 
both urns are equally likely to be chosen. 
The students then come to the front of the room one by one. 
Each student draws a ball at random from the urn and puts it back without showing it to the rest of the class. 
The student then has to guess the majority color of the urn, announcing her guess publicly. 
Each student thus makes their decision based on their draw and the announcements of those gone before them.

Let us consider how such an experiment proceeds. 
The first student only has her own draw to go by, so she will announce the drawn color as her best guess of the majority color. 
The second student knows this, 
so together with her own draw she has two independent draws as information. 
If the colors of the two agree, then the second student announces this color. 
If the colors of the two draws differ, then she has to use a tie-breaking rule---let us assume that she breaks ties by following her own draw. 
With this choice we see that the second student also announces the color of the ball she drew. 
Hence the third student has three independent draws as information and her best guess for the majority color of the urn will be the majority color among the three draws.

Notice that if the first two announced colors were blue, then the third student announces blue \emph{regardless of the color of her draw}. 
The fourth student knows this and hence the announcement of the third student has no information value. 
The fourth student is thus in the same situation as the third one and will also just announce blue following the first two students. 
Following the same logic, all subsequent students will announce blue, regardless of the color of their draw. 
This phenomenon is known as \emph{herding} or as an \emph{information cascade}.
Its study was originated by Banerjee~\cite{banerjee1992simple} and by Bikhchandani, Hirshleifer, and Welch~\cite{bikhchandani1992theory}, independently and concurrently;
we refer to Easley and Kleinberg~\cite[Chapter~16]{easley2010networks} for an exposition.

The main issue with information cascades is that they can be \emph{wrong}.
For instance, it might be that the urn is majority yellow, but the first two students draw a blue ball and announce blue, and hence \emph{everyone} announces blue as their guess. 
The main cause of this is that information cascades can be based on very little information: the actions of a few initial actors can determine all subsequent actions.
This also explains why information cascades are \emph{fragile}: if additional information is revealed (e.g., someone reveals not only their action but also their private signal) or if some people deviate from rational behavior, then wrong information cascades can be broken. The focus of this paper is to analyze the fragility of information cascades quantitatively.

To this end, we 
study a variant of the simple model of sequential decision making studied previously, where not all agents are Bayesian: 
some are ``revealers'', who disregard the actions of others and act solely based on their private signal. 
This is motivated by both empirical results from laboratory experiments on human behavior in such a setting~\cite{anderson1997information}, as well as theoretical considerations~\cite{bernardo2001evolution}; see Section~\ref{sec:related} for further discussion of related work. 
We assume that the player at time $t$ is a revealer with probability $p_{t}$, independently of everything else, and is a Bayesian otherwise. 
While agents do not know whether those before them were Bayesians or revealers, this process still introduces additional information that can be useful for making inferences. 
Are wrong information cascades broken in such a model? That is, do people eventually learn the ``correct'' action?

We show that the answer is yes: there exist revealing probabilities $\left\{ p_{t} \right\}_{t=1}^{\infty}$ such that learning occurs.
Moreover, we study the optimal asymptotic rate at which the error probability at time $t$ can go to zero.
We show that the optimal policy is for the player at time $t$ to follow their private information with probability $p_{t} = c/t$, leading to a learning rate of $c'/t$, where the constants $c$ and $c'$ are explicit.

\subsection{Model and main result} \label{sec:model_main_result} 

We describe the simplest case of the model first, in order to focus on the conceptual points; we discuss generalizations at the end of the paper.
The state of the world is
$\theta \in \left\{ 1, 2 \right\}$,
chosen uniformly at random, that is,
$\p \left( \theta = 1 \right) = \p \left( \theta = 2 \right) = 1/2$.
At times $t = 1, 2, 3, \dots$ players try to guess the state of the world,
based on their private information, as well as observing the actions (guesses) of those before them.

The private signals are drawn in the following way.
There is an urn that contains two types of balls: type $1$ balls are blue and type $2$ balls are yellow.
Given $\theta$, there are $a$ balls of type $\theta$ in the urn and $b$ balls of the other type,
where we assume that $a > b > 0$.
Each player draws a single ball (with replacement) from the urn, its color is their private signal.
In other words, the private signals
$X_{1}, X_{2}, \dots$ are i.i.d.\ with the following distribution:
\begin{align*}
 \p \left( X_{1} = 1 \, \middle| \, \theta = 1 \right) &= \frac{a}{a+b},
 \qquad
 \qquad
 \p \left( X_{1} = 2 \, \middle| \, \theta = 1 \right) = \frac{b}{a+b},
 \\
 \p \left( X_{1} = 1 \, \middle| \, \theta = 2 \right) &= \frac{b}{a+b},
 \qquad
 \qquad
 \p \left( X_{1} = 2 \, \middle| \, \theta = 2 \right) = \frac{a}{a+b}.
\end{align*}
The goal of the players is to guess the majority color (type) of the balls in the urn.
We denote the actions (guesses) of the players by $Z_{1}, Z_{2}, \dots$.
We assume that each player is one of two kinds:
\begin{itemize}
 \item a \emph{Bayesian}, whose guess is the maximum a posteriori (MAP) estimate\footnote{We assume that if the posteriors are equal, then a Bayesian follows their private signal.}; or
 \item a \emph{revealer}, whose guess is their private signal.
\end{itemize}
We assume that player $t$ is a revealer with probability $p_{t}$, independently of everything else.
Formally, let $I_{1}, I_{2}, \dots$ be independent Bernoulli random variables 
(and also independent of everything else) 
such that $\E \left[ I_{t} \right] = p_{t}$.
If $I_{t} = 0$, then player $t$ is a Bayesian and hence $Z_{t} = \MAP \left( Z_{1}, \dots, Z_{t-1}, X_{t} \right)$,
while if $I_{t} = 1$, then player $t$ is a revealer and hence $Z_{t} = X_{t}$.
Note that players do not know whether the players before them are Bayesians or revealers.
We do assume, however, that the players know the probabilities $\left\{ p_{t} \right\}_{t = 1}^{\infty}$.

The players aim to learn the majority color/type of the urn, that is, to learn $\theta$, and also to minimize the probability of an incorrect guess.
Denote by
\[
\tE_{t} := \p \left( Z_{t} \neq \theta \right)
\]
the probability that the guess of player $t$ is incorrect.
We aim to understand the optimal asymptotic rate at which the error probability $\tE_{t}$ can go to zero;
the following theorem is our main result.

\begin{theorem}\label{thm:main_two}
Consider the model described above and let
\begin{equation}\label{eq:kappastar}
\kappa_{\star} \equiv \kappa_{\star} \left( a, b \right) := \frac{1}{1+\frac{a/b-1}{\log a/b}\left( \log \left( \frac{a/b-1}{\log a/b} \right) - 1 \right)}.
\end{equation}
We have that
\begin{equation}\label{eq:opt_rate_two}
 \inf_{\left\{ p_{t} \right\}_{t=1}^{\infty}} \limsup_{t \to \infty} t \tE_{t}
 = \kappa_{\star} \left( a, b \right).
\end{equation}
\end{theorem}
That is, the optimal rate of learning is $1/t$, and we obtain the specific constant as well in~\eqref{eq:kappastar}.
As we shall see, one can get arbitrarily close to the optimum by taking
\begin{equation}\label{eq:pt_opt_main}
 p_{t} = \left( 1 + \eps \right) \frac{a+b}{b} \cdot \frac{\kappa_{\star} \left( a, b \right)}{t} \wedge 1
\end{equation}
for $t \geq 1$, where $\eps > 0$ is arbitrary (and where we use the standard notation $x \wedge y := \min \left\{ x, y \right\}$).

\subsection{Heuristic explanation of the optimal rate of learning} \label{sec:heuristics} 

We now provide intuition for why $1/t$ is the optimal order for the rate of learning, as well as the reasons behind the constant in~\eqref{eq:kappastar}. 
First, note that the probability that player $t$ is a revealer and draws a ball of the minority color is $\tfrac{b}{a+b} p_{t}$. 
When this occurs, player $t$ guesses incorrectly, implying that $\tE_{t} \geq \tfrac{b}{a+b} p_{t}$. 
So in order for the error probability $\tE_{t}$ to go to zero, $p_{t}$ must go to zero as $t \to \infty$. 

On the other hand, $p_{t}$ cannot go to zero too quickly. 
If $\sum_{t=1}^{\infty} p_{t} < \infty$, then by the Borel-Cantelli lemma there will be only finitely many revealers almost surely. 
This leads to a situation similar to when there are no revealers: 
if a correct cascade has not started before the last revealer, then there is a constant probability of ending up in a wrong cascade.

In fact, $p_t$ should decay as $1/t$ to achieve the optimal rate of learning. 
To see a lower bound of this order, let $p_{t} = \delta / t$ for $\delta$ small. 
By a Chernoff bound, with high probability there will be at most $2 \delta \log t$ revealers among the first $t$ players. 
If the first two players and all revealers until time $t$ draw balls of the minority color, then so will every player until time $t$. 
This event has probability at least $c \left( b / (a+b) \right)^{2 \delta \log t}$ for some constant $c$, 
which is greater than $t^{-\eps}$ if $\delta > 0$ is small enough.

To see an upper bound, we now argue that if $p_{t} = C / t$ with $C$ large enough, 
then the probability that a wrong cascade (among Bayesians) lasts until time $t$ is $o(1/t)$. 
Indeed, in such a cascade some of the revealers will be visible---precisely those that deviate from the cascade consensus. 
The player at time $k$ has probability $\tfrac{a}{a+b} p_{k}$ to be a revealer who draws the majority color, 
and probability $\tfrac{b}{a+b} p_{k}$ to be a revealer who draws the minority color. 
Hence in a wrong cascade there will be, in expectation, $\tfrac{a}{a+b} C \log t$ deviations from the cascade consensus by time $t$, 
while in a right cascade the expected number of deviations is only $\tfrac{b}{a+b} C \log t$. 
The total number of deviations by time $t$ will roughly be Poisson distributed. 
The probability that a $\Poi \left( \lambda \right)$ random variable is larger by a constant factor than its mean is exponentially small in $\lambda$ 
and here $\lambda$ is on the order of $\log t$. 
Hence by taking $C$ large this exponential in $C \log t$ will be $o(1/t)$.

In fact, the heuristics of the previous paragraph give the right constant as well. 
To distinguish between right and wrong cascades we need to distinguish between $\Poi \left( \tfrac{a}{a+b} C \log t \right)$ and $\Poi \left( \tfrac{b}{a+b} C \log t \right)$ random variables. The total variation distance between them satisfies 
\begin{equation}\label{eq:tv_poisson}
 1 - \TV \left( \Poi \left( \tfrac{a}{a+b} C \log t \right), \Poi \left( \tfrac{b}{a+b} C \log t \right) \right) 
 = 
 t^{- \left( 1 + o \left( 1 \right) \right) f \left( a, b, C \right)}
\end{equation}
for some (explicit) function $f \left(a, b, C \right)$. 
The right hand side of~\eqref{eq:tv_poisson} is roughly the error probability if player $t$ is a Bayesian. 
This term 
should be balanced with the term $\tfrac{b}{a+b} p_{t} = \tfrac{bC}{a+b} \cdot \tfrac{1}{t}$ coming from player $t$ being a revealer and drawing a ball of the minority color. 
This balancing requires choosing $C = C(a,b)$ such that $f \left( a, b, C \right) = 1$, 
which occurs when $C = \tfrac{a+b}{b} \kappa_{\star} \left( a, b \right)$, just as in~\eqref{eq:pt_opt_main}.


\subsection{Related work} \label{sec:related} 

The special case of the model described in the previous subsection where every agent is Bayesian (i.e., with $p_{t} = 0$ for every $t \geq 1$) is identical to the model described in the exposition of Easley and Kleinberg~\cite[Chapter~16]{easley2010networks}.
The original model of Bikhchandani, Hirshleifer, and Welch~\cite{bikhchandani1992theory} differs only in its tie-breaking rule (breaking ties by flipping a fair coin), while that of Banerjee~\cite{banerjee1992simple} differs in the signal distribution (false signals are drawn from a continuous distribution).
Despite these minor differences, these models all share the same phenomenological behavior as described in the introductory paragraphs.

In particular, Bikhchandani~et~al.~\cite{bikhchandani1992theory} emphasize the \emph{fragility} of information cascades with respect to different types of shocks, as prior work on conforming behavior could not explain this phenomenon.
They show examples from numerous fields (e.g., politics, zoology, medicine, and finance) where cascades occur and are fragile.
The current paper can be viewed as a more detailed quantitative exploration of the fragility of cascades. What amount of additional information is needed to break wrong  cascades? What is the optimal rate of learning that can be achieved?

One possible source of additional information comes from people not acting in a rational, Bayesian manner.
It is well documented that human behavior is often irrational (see, e.g.,~\cite{kahneman1973psychology}).
In the information cascades setting, laboratory experiments by Anderson and Holt~\cite{anderson1997information} show that
while most participants act rationally, many do not.
When deviations from rational behavior occur,
participants often act mainly or solely based on their private information, disregarding the information in the actions of those before them.\footnote{See also related experiments and results by \c{C}elen and Kariv~\cite{ccelen2004distinguishing} for a setting with continuous signals.}
Such individuals effectively reveal their private signal, which is valuable information for those coming after them.
The model described in Section~\ref{sec:model_main_result}, which contains \emph{Bayesians} and \emph{revealers}, captures this empirically observed behavioral phenomenon.

A closely related model was introduced and studied by Bernardo and Welch~\cite{bernardo2001evolution}.
This model also contains two types of individuals:
(1) rational ones and
(2) overconfident ones, termed ``entrepreneurs'', who put more weight on their private signal than a rational individual would.
As the authors mention in their paper, their motivation was not to show that information cascades can be broken by overconfident behavior,
but rather to offer a simple explanation for the existence of overconfident individuals based on group selection principles.
Nevertheless, since here our focus is on breaking wrong information cascades, we compare our work with theirs in this regard.

In~\cite{bernardo2001evolution} the overconfidence of entrepreneurs is termed ``modest'' if they still put positive weight on the information from individuals before them,
and it is termed ``extreme'' if they act solely based on their private signal.
If the overconfidence of entrepreneurs is modest, then still a wrong cascade occurs with positive probability, bounded away from zero, which is undesirable.
Only if the overconfidence of entrepreneurs is extreme---as in the model in Section~\ref{sec:model_main_result}---can learning occur eventually with probability one.
Bernardo and Welch study their model via simulations which, in the extreme overconfidence setting, suggest that a vanishing fraction of entrepreneurs is optimal.
Our work is a rigorous and much more detailed study of this model\footnote{We note that there are small differences in the model studied here and the model of Bernardo and Welch~\cite{bernardo2001evolution}; for instance, in~\cite{bernardo2001evolution} it is assumed that the identities of entrepreneurs are known whereas we do not make this assumption.}; in particular, our results imply that the optimal number of entrepreneurs is \emph{logarithmic} in the size of the group.

The recent work of Cheng, Hann-Caruthers, and Tamuz~\cite{chengdeterministic} also considers sequential learning models with non-Bayesian agents 
and shows that wrong cascades can be avoided if there are some non-Bayesian agents. 
However, they assume (like in~\cite{bernardo2001evolution}) that each agent knows which of the previous agents were revealers, an assumption that we do not make. 
More importantly, the main contribution of the current paper is the explicit characterization of the optimal rate of learning.

Another possible source of additional information is using the first few agents as ``guinea pigs'', that is, forcing them to follow their private signals; see the work of Sgroi~\cite{sgroi2002optimizing}. 
Le, Subramanian, and Berry~\cite{le2017information} point out that this is related to the multi-armed bandit literature, 
with guinea pigs corresponding to agents used for exploring~\cite{lai1985asymptotically,agrawal1995sample}. 
They also mention that it follows from this literature that the optimal number of guinea pigs is logarithmic in the number of agents~\cite{agrawal1995sample}; 
this is consistent with our results, albeit in a slightly different setting.

The framework for sequential decision making described in this paper assumes finite discrete private signals.
If the informativeness of private signals is unbounded (e.g., Gaussian signals), then wrong cascades do not form and asymptotic learning occurs~\cite{smith2000pathological}.
In such settings the main question concerns the speed of asymptotic learning, 
see, for instance, the work of Hann-Caruthers, Martynov, and Tamuz~\cite{HCMT17}.

The framework of this paper also fits into the broader field of social learning.
In particular, there is a large literature on learning in social networks.
Acemoglu, Dahleh, Lobel, and Ozdaglar~\cite{acemoglu2011bayesian} consider a model of sequential decision making where agents act only once,
but each agent can only observe a \emph{subset} of previous actions, based on a stochastic social network.
One of their results is that asymptotic learning occurs even if private signals have bounded informativeness
if there are sufficiently many individuals whose neighborhoods are non-persuasive
and hence whose action will necessarily be influenced by their private signal.
Similar to revealers in the model described in Section~\ref{sec:model_main_result},
these individuals provide a sufficient amount of information for those coming after them
to lead to asymptotic learning.

Another typical setting that is studied involves agents who take repeated actions based on their private signal, as well as observing the actions of their neighbors in the network.
The main questions include whether or not all agents learn the correct action eventually,
what is the speed of learning if it occurs, and how do these depend on the network topology.
We highlight recent work of Harel, Mossel, Strack, and Tamuz~\cite{HMST17},
which is similar in spirit to the current paper in that it provides a detailed study of the asymptotic rates of social learning in a mean field setting.
A complete overview of the literature is beyond the scope of this article;
we refer the reader to the two papers above, as well as to the works of Gale and Kariv~\cite{gale2003bayesian}, Mossel, Sly, and Tamuz~\cite{MST12,MST14,mossel2015strategic}, and the references therein for more.

\section{Proof of Theorem~\ref{thm:main_two}} \label{sec:proofs_two} 

The action of player $t$ can be wrong in two ways:
(i) if they act as a Bayesian and the MAP estimator is incorrect, or
(ii) if they act on only their private signal and their draw from the urn is the minority type/color.
Hence, conditioning on the coin flip deciding whether player $t$ is a Bayesian or a revealer, we obtain that
\[
\tE_{t} = \p \left( Z_{t} \neq \theta \, \middle| \, I_{t} = 0 \right) \left( 1 - p_{t} \right) + \p \left( Z_{t} \neq \theta \, \middle| \, I_{t} = 1 \right) p_{t}.
\]
Now recall that given $I_{t} = 1$, we have $Z_{t} = X_{t}$, and so
\[
 \p \left( Z_{t} \neq \theta \, \middle| \, I_{t} = 1 \right) = \p \left( X_{t} \neq \theta \right) = \frac{b}{a+b}.
\]
Recall also that given $I_{t} = 0$, we have $Z_{t} = \MAP \left( Z_{1}, \dots, Z_{t-1}, X_{t} \right)$.
Thus 
\begin{equation}\label{eq:error_two_parts}
\tE_{t} = \p \left( \MAP \left( Z_{1}, \dots, Z_{t-1}, X_{t} \right) \neq \theta \right) \left( 1 - p_{t} \right) + \frac{b}{a+b} p_{t}.
\end{equation}
So we need to understand the probability that the MAP estimator is incorrect at time $t$. 
We summarize the behavior of the MAP estimator in Lemma~\ref{lem:map_estimator_two} and then prove Theorem~\ref{thm:main_two} using this, before turning to the proof of the lemma. 
In the statement of the lemma and throughout the paper we use standard asymptotic notation; 
for instance $f(t) = o(g(t))$ as $t \to \infty$ if $\lim_{t\to\infty} f(t) / g(t) = 0$ 
and $f(t) = \omega(g(t))$ as $t\to \infty$ if $\lim_{t \to \infty} f(t) / g(t) = \infty$.

\begin{lemma}\label{lem:map_estimator_two} 
Consider the setting of Theorem~\ref{thm:main_two} and fix $\eps > 0$. 
\begin{enumerate}[(a)]
\item\label{lem:map_ub} Suppose that 
\begin{equation}\label{eq:pt_opt}
 p_{t} = \left( 1 + \eps \right) \frac{a+b}{b} \cdot \frac{\kappa_{\star} \left( a, b \right)}{t} \wedge 1
\end{equation}
for every $t \geq 1$. Then 
\begin{equation}\label{eq:goal_o1t}
 \qquad \qquad 
 \p \left( \MAP \left( Z_{1}, \dots, Z_{t-1}, X_{t} \right) \neq \theta \right) = o \left( \frac{1}{t} \right) 
 \qquad 
 \text{ as } t \to \infty.
\end{equation}
\item\label{lem:map_lb} Suppose that  
\begin{equation}\label{eq:pt_lim}
 \limsup_{t \to \infty} t  p_{t} \leq \left( 1 - \eps \right) \frac{a+b}{b} \kappa_{\star} \left( a, b \right). 
\end{equation}
Then 
\begin{equation}\label{eq:goal_omega1t}
 \qquad \qquad 
 \p \left( \MAP \left( Z_{1}, \dots, Z_{t-1}, X_{t} \right) \neq \theta \right) = \omega \left( \frac{1}{t} \right)
 \qquad 
 \text{ as } t \to \infty.
\end{equation}
\end{enumerate}
\end{lemma}

\begin{proof}[Proof of Theorem~\ref{thm:main_two}]
Choose $\left\{ p_{t} \right\}_{t=1}^{\infty}$ as in~\eqref{eq:pt_opt}.
Lemma~\ref{lem:map_estimator_two} says that then~\eqref{eq:goal_o1t} holds,
and hence by~\eqref{eq:error_two_parts} we have that
$\tE_{t} =  \left( 1 + o \left( 1 \right) \right) \left( 1 + \eps \right) \kappa_{\star} \left( a, b \right)/t$
as $t \to \infty$.
Since $\eps > 0$ is arbitrary, we have that
\begin{equation}\label{eq:opt_UB}
 \inf_{\left\{ p_{t} \right\}_{t=1}^{\infty}} \limsup_{t \to \infty} t \tE_{t} \leq \kappa_{\star} \left( a, b \right).
\end{equation}

To show that this is optimal, first note that $\tE_{t} \geq \tfrac{b}{a+b} p_{t}$.
Hence, if
\begin{equation}\label{eq:beating_opt}
 \limsup_{t \to \infty} t \tE_{t} < \kappa_{\star} \left( a, b \right),
\end{equation}
then the corresponding sequence of probabilities $\left\{ p_{t} \right\}_{t=1}^{\infty}$ must satisfy~\eqref{eq:pt_lim} for some $\eps > 0$.
But then Lemma~\ref{lem:map_estimator_two} says that
$\p \left( \MAP \left( Z_{1}, \dots, Z_{t-1}, X_{t} \right) \neq \theta \right) = \omega \left( 1/t \right)$. 
So by~\eqref{eq:error_two_parts} we have that
$\tE_{t} = \omega \left( 1 / t \right)$, which contradicts~\eqref{eq:beating_opt}. Thus the inequality in~\eqref{eq:opt_UB} is, in fact, an equality.
\end{proof}

The rest of this section consists of the proof of Lemma~\ref{lem:map_estimator_two}. 
We start in Section~\ref{sec:map} by introducing notation and making basic observations about the MAP estimator that are useful for both bounds in Lemma~\ref{lem:map_estimator_two}. 
Then we turn to the proof of Lemma~\ref{lem:map_estimator_two}~\eqref{lem:map_ub} in Section~\ref{sec:ub_proof} 
and we conclude with the proof of Lemma~\ref{lem:map_estimator_two}~\eqref{lem:map_lb} in Section~\ref{sec:lb_proof}.

\subsection{The MAP estimator}\label{sec:map} 

In this subsection we introduce some notation and make basic observations about the MAP estimator that are useful for the bounds in Lemma~\ref{lem:map_estimator_two}, 
which is proven subsequently.

For $i \in \left\{ 1, 2 \right\}$, let $\p_{i}$ denote the probability measure conditioned on $\theta = i$, that is, 
$\p_{i} \left( \cdot \right) := \p \left( \cdot \, \middle| \, \theta = i \right)$. 
Similarly, $\E_{i}$ denotes expectation conditioned on $\theta = i$. 
For $i \in \left\{ 1, 2 \right\}$ and $t \geq 1$, denote by $P_{i}^{t}$ the distribution of
$\left( Z_{1}, \dots, Z_{t} \right)$.
That is, for $\left( z_{1}, \dots, z_{t} \right) \in \left\{ 1, 2 \right\}^{t}$, let 
\[
 P_{i}^{t} \left( z_{1}, \dots, z_{t} \right) := \p_{i} \left( Z_{1} = z_{1}, \dots, Z_{t} = z_{t} \right).
\]
Similarly, for $i \in \left\{ 1, 2 \right\}$ and $t \geq 1$, denote by $Q_{i}^{t}$ the distribution of
$\left( Z_{1}, \dots, Z_{t-1}, X_{t} \right)$.
Define also the corresponding likelihoods:
\begin{align*}
 L_{i}^{t} &:= P_{i}^{t} \left( Z_{1}, \dots, Z_{t} \right), \\
 D_{i}^{t} &:= Q_{i}^{t} \left( Z_{1}, \dots, Z_{t-1}, X_{t} \right),
\end{align*}
with $D_{i}^{0} = L_{i}^{0} = 1$ for $i \in \left\{ 1, 2 \right\}$.
An outside observer who records the actions of the first $t$ players can compute the likelihoods $L_{1}^{t}$ and $L_{2}^{t}$,
while player $t$ can compute the likelihoods $D_{1}^{t}$ and $D_{2}^{t}$.
If player $t$ is a Bayesian, then their guess is based on the likelihoods $D_{1}^{t}$ and $D_{2}^{t}$.
Specifically, since the prior on $\theta$ is uniform, we have that
\begin{equation}\label{eq:map_details}
 \MAP \left( Z_{1}, \dots, Z_{t-1}, X_{t} \right) =
 \begin{cases}
  1, & \text{ if } D_{1}^{t} > D_{2}^{t}, \\
  2, & \text{ if } D_{1}^{t} < D_{2}^{t}, \\
  X_{t}, & \text{ if } D_{1}^{t} = D_{2}^{t},
 \end{cases}
\end{equation}
where the last line is due to the tie-breaking rule; 
recall that if the posteriors are equal then a Bayesian follows their private signal.

For $i \in \left\{ 1, 2 \right\}$ and $x \in \left\{ 1, 2 \right\}$, define
\[
 \varphi_{i} \left( x \right) := \frac{a}{a+b} \mathbf{1}_{\left\{ x = i \right\}} + \frac{b}{a+b} \mathbf{1}_{\left\{ x \neq i \right\}},
\]
and note that, since $X_{t}$ is independent of everything else, we have that 
$D_{i}^{t} = L_{i}^{t-1} \varphi_{i} \left( X_{t} \right)$ 
for $i \in \left\{ 1, 2 \right\}$ and $t \geq 1$.
Thus in order to understand the likelihoods $D_{1}^{t}$ and $D_{2}^{t}$, we need to analyze $L_{1}^{t}$ and $L_{2}^{t}$.
Define the $L$-likelihood and the $D$-likelihood ratios as
\[
 R_{t} := \frac{L_{1}^{t}}{L_{2}^{t}} \qquad \text{ and } \qquad R_{t}' := \frac{D_{1}^{t}}{D_{2}^{t}},
\]
respectively.
We can write
\begin{equation}\label{eq:likelihood_ratio_D}
 R_{t}' = R_{t-1} \frac{\varphi_{1} \left( X_{t} \right)}{\varphi_{2} \left( X_{t} \right)}
\end{equation}
and hence we can determine the action of player $t$ for given $R_{t-1}$ and $X_{t}$.
Note that the random variable
$\varphi_{1} \left( X_{t} \right) / \varphi_{2} \left( X_{t} \right)$
takes values in $\left\{ b/a, a/b \right\}$,
and hence we have the following three cases.
\begin{itemize}
 \item If $R_{t-1} < b/a$, then $R_{t}' < 1$ and hence $D_{1}^{t} < D_{2}^{t}$, regardless of the value of $X_{t}$.
 Hence $Z_{t} = 1$ if player $t$ is a revealer and $X_{t} = 1$, and $Z_{t} = 2$ otherwise.
 \item If $R_{t-1} \in \left[ b/a, a/b \right]$, then $Z_{t} = X_{t}$.
 This can be checked by considering both cases.
 If $X_{t} = 1$ then $\varphi_{1} \left( X_{t} \right) / \varphi_{2} \left( X_{t} \right) = a/b$.
 Therefore by~\eqref{eq:likelihood_ratio_D} we have that $R_{t}' \geq 1$ and hence $D_{1}^{t} \geq D_{2}^{t}$.
 If $D_{1}^{t} > D_{2}^{t}$ then $Z_{t} = 1 = X_{t}$ by the definition of the MAP estimator,
 while if $D_{1}^{t} = D_{2}^{t}$ then $Z_{t} = X_{t}$ by the tie-breaking rule.
 The case of $X_{t} = 2$ is analogous.
 \item If $R_{t-1} > a/b$, then $R_{t}' > 1$ and hence $D_{1}^{t} > D_{2}^{t}$, regardless of the value of $X_{t}$.
 Hence $Z_{t} = 2$ if player $t$ is a revealer and $X_{t} = 2$, and $Z_{t} = 1$ otherwise.
\end{itemize}
The three cases above describe how the action of a player depends on their private signal and on the actions of those who acted before them.
This allows us to analyze how 
the $L$-likelihood ratio 
$R_{t}$ evolves.

The probability that the MAP estimator makes an error at time $t$ can be expressed using the $D$-likelihood ratio as follows.
To abbreviate the notation for vectors, we write $z_{1}^{t} \equiv \left( z_{1}, \dots, z_{t} \right)$.
First, conditioning on the value of $\theta$ we obtain that
\begin{equation}\label{eq:map_error_theta}
\p \left( \MAP \left( Z_{1}^{t-1}, X_{t} \right) \neq \theta \right)
=
\frac{1}{2} \p_{1} \left( \MAP \left( Z_{1}^{t-1}, X_{t} \right) = 2 \right)
+
\frac{1}{2} \p_{2} \left( \MAP \left( Z_{1}^{t-1}, X_{t} \right) = 1 \right).
\end{equation}
The two terms on the right hand side of~\eqref{eq:map_error_theta} are equal due to symmetry, so
\begin{equation}\label{eq:map_error_theta1}
\p \left( \MAP \left( Z_{1}^{t-1}, X_{t} \right) \neq \theta \right)
=
\p_{1} \left( \MAP \left( Z_{1}^{t-1}, X_{t} \right) = 2 \right).
\end{equation}
Using~\eqref{eq:map_details} we obtain the following upper and lower bounds:
\[
\p_{1} \left( R_{t}' < 1 \right)
\leq
\p \left( \MAP \left( Z_{1}^{t-1}, X_{t} \right) \neq \theta \right)
\leq
\p_{1} \left( R_{t}' \leq 1 \right).
\]
It is more convenient to work with the $L$-likelihood ratio, so using the fact that
$\tfrac{b}{a} R_{t-1} \leq R_{t}' \leq \tfrac{a}{b} R_{t-1}$
we obtain that
\begin{equation}\label{eq:map_error_bounds}
\p_{1} \left( R_{t-1} < \frac{b}{a} \right)
\leq
\p \left( \MAP \left( Z_{1}^{t-1}, X_{t} \right) \neq \theta \right)
\leq
\p_{1} \left( R_{t-1} \leq \frac{a}{b} \right).
\end{equation}
To obtain parts~\eqref{lem:map_ub} and~\eqref{lem:map_lb} of Lemma~\ref{lem:map_estimator_two} we bound from above and below the probabilities appearing in~\eqref{eq:map_error_bounds}.

\subsection{An upper bound} \label{sec:ub_proof} 

\begin{proof}[Proof of Lemma~\ref{lem:map_estimator_two}~\eqref{lem:map_ub}]
By~\eqref{eq:map_error_bounds} our goal is to show that
\begin{equation}\label{eq:ub_goal}
\p_{1} \left( R_{t} \leq \frac{a}{b} \right) = o \left( \frac{1}{t} \right)
\end{equation}
as $t \to \infty$, and recall that we assume that the revealing probabilities $\left\{ p_{t} \right\}_{t=1}^{\infty}$ are as in~\eqref{eq:pt_opt}.

Let $\left\{ \mathcal{F}_{t} \right\}_{t \geq 0}$ denote the filtration defined by the random variables $\left\{ Z_{t} \right\}_{t \geq 1}$.
Observe that, given $\theta = 1$, the inverse of the $L$-likelihood ratio, $\left\{ 1 / R_{t} \right\}_{t \geq 0}$, is a martingale with respect to $\left\{ \mathcal{F}_{t} \right\}_{t \geq 0}$.
In particular, this implies that
$\E_{1} \left[ R_{t}^{-1} \right] = 1$.
Since $x \mapsto x^{\lambda}$ is a concave function for $x \in \left( 0, \infty \right)$ when $\lambda \in \left[ 0, 1 \right]$, we have that,
given $\theta = 1$, the sequence $\left\{ R_{t}^{-\lambda} \right\}_{t \geq 0}$ is a supermartingale with respect to $\left\{ \mathcal{F}_{t} \right\}_{t \geq 0}$.
Thus 
$\E_{1} \left[ R_{t}^{-\lambda} \, \middle| \, \mathcal{F}_{t-1} \right]
\leq
R_{t-1}^{-\lambda}$. 
We now compute the conditional expectation explicitly: 
\begin{align}
\E_{1} \left[ R_{t}^{-\lambda} \, \middle| \, \mathcal{F}_{t-1} \right]
&=
R_{t-1}^{-\lambda} \sum_{i \in \left\{ 1, 2 \right\}} \p_{1} \left( Z_{t} = i \, \middle| \, \mathcal{F}_{t-1} \right) \left( \frac{\p_{1} \left( Z_{t} = i \, \middle| \, \mathcal{F}_{t-1} \right)}{\p_{2} \left( Z_{t} = i \, \middle| \, \mathcal{F}_{t-1} \right)} \right)^{-\lambda} \notag \\
&=
R_{t-1}^{-\lambda} \sum_{i \in \left\{ 1, 2 \right\}} \p_{1} \left( Z_{t} = i \, \middle| \, \mathcal{F}_{t-1} \right)^{1- \lambda} \p_{2} \left( Z_{t} = i \, \middle| \, \mathcal{F}_{t-1} \right)^{\lambda}. \label{eq:cond_exp_lambda}
\end{align}
The values of the conditional probabilities in~\eqref{eq:cond_exp_lambda} depend on the value of $R_{t-1}$. As described previously, we have three cases:
\begin{equation}\label{eq:p1Zt1}
 \p_{1} \left( Z_{t} = 1 \, \middle| \, \mathcal{F}_{t-1} \right)
 =
 \begin{cases}
  \frac{a}{a+b} p_{t} & \text{ if } R_{t-1} < \frac{b}{a}, \\
  \frac{a}{a+b} & \text{ if } R_{t-1} \in \left[ \frac{b}{a}, \frac{a}{b} \right], \\
  1 - \frac{b}{a+b} p_{t} & \text{ if } R_{t-1} > \frac{a}{b},
 \end{cases}
\end{equation}
and also
\begin{equation}\label{eq:p2Zt1}
 \p_{2} \left( Z_{t} = 1 \, \middle| \, \mathcal{F}_{t-1} \right)
 =
 \begin{cases}
  \frac{b}{a+b} p_{t} & \text{ if } R_{t-1} < \frac{b}{a}, \\
  \frac{b}{a+b} & \text{ if } R_{t-1} \in \left[ \frac{b}{a}, \frac{a}{b} \right], \\
  1 - \frac{a}{a+b} p_{t} & \text{ if } R_{t-1} > \frac{a}{b}.
 \end{cases}
\end{equation}
Plugging these back into~\eqref{eq:cond_exp_lambda} we obtain the conditional expectation in the three cases:
\begin{align}
\E_{1} \left[ \left( \frac{R_{t}}{R_{t-1}} \right)^{-\lambda} \, \middle| \, R_{t-1} < \frac{b}{a} \right]
&= \frac{a^{1-\lambda} b^{\lambda}}{a+b} p_{t} + \left( 1 - \frac{a}{a+b} p_{t} \right)^{1 - \lambda} \left( 1 - \frac{b}{a+b} p_{t} \right)^{\lambda}, \label{eq:cond_exp_small} \\
\E_{1} \left[ \left( \frac{R_{t}}{R_{t-1}} \right)^{-\lambda} \, \middle| \, R_{t-1} \in \left[ \frac{b}{a}, \frac{a}{b} \right] \right]
&= \frac{a^{1-\lambda} b^{\lambda} + a^{\lambda} b^{1 - \lambda}}{a+b}, \label{eq:cond_exp_medium} \\
\E_{1} \left[ \left( \frac{R_{t}}{R_{t-1}} \right)^{-\lambda} \, \middle| \, R_{t-1} > \frac{a}{b} \right]
&= \frac{a^{\lambda} b^{1-\lambda}}{a+b} p_{t} + \left( 1 - \frac{b}{a+b} p_{t} \right)^{1- \lambda} \left( 1 - \frac{a}{a+b} p_{t} \right)^{\lambda}. \label{eq:cond_exp_big}
\end{align}
The right hand side of~\eqref{eq:cond_exp_medium} is strictly less than $1$,
while the right hand sides of~\eqref{eq:cond_exp_small} and~\eqref{eq:cond_exp_big} converge to $1$ as $t \to \infty$.
To estimate the quantities in~\eqref{eq:cond_exp_small} and~\eqref{eq:cond_exp_big}, note that
$\left( 1 - \delta \right)^{\lambda} = 1 - \lambda \delta + \Theta \left( \delta^{2} \right)$ as $\delta \to 0$.
Defining
\[
 f_{\lambda} \equiv f_{\lambda} \left( a, b \right) := \frac{\lambda a + \left( 1 - \lambda \right) b - a^{\lambda} b^{1 - \lambda}}{a+b},
\]
we have from~\eqref{eq:cond_exp_small} --- \eqref{eq:cond_exp_big} that
\begin{align}
\E_{1} \left[ \left( \frac{R_{t}}{R_{t-1}} \right)^{-\lambda} \, \middle| \, R_{t-1} < \frac{b}{a} \right]
&= 1 - f_{1 - \lambda} \left( a, b \right) p_{t} + O \left( p_{t}^{2} \right), \label{eq:cond_exp_small_approx} \\
\E_{1} \left[ \left( \frac{R_{t}}{R_{t-1}} \right)^{-\lambda} \, \middle| \, R_{t-1} \in \left[ \frac{b}{a}, \frac{a}{b} \right] \right]
&=1-\left(f_{\lambda}(a,b)+f_{1-\lambda}(a,b)\right) \label{eq:cond_exp_med_approx}\\
\E_{1} \left[ \left( \frac{R_{t}}{R_{t-1}} \right)^{-\lambda} \, \middle| \, R_{t-1} > \frac{a}{b} \right]
&= 1 - f_{\lambda} \left( a, b \right) p_{t} + O \left( p_{t}^{2} \right), \label{eq:cond_exp_big_approx}
\end{align}
as $t \to \infty$.
On the interval $\lambda \in \left[ 0, 1 \right]$ the function $\lambda \mapsto f_{\lambda}$ is concave and nonnegative with $f_{0} = f_{1} = 0$,
it attains its maximum at
\[
 \lambda_{\star} \equiv \lambda_{\star} \left( a, b \right) := \frac{\log \left( \frac{a/b - 1}{\log a/b} \right)}{\log a/b},
\]
and its maximum value is
\begin{equation}\label{eq:max_value}
 f_{\lambda_{\star}} \left( a, b \right) = \frac{b}{a+b} \cdot \frac{1}{\kappa_{\star} \left( a, b \right)},
\end{equation}
where recall the definition of $\kappa_{\star}$ from~\eqref{eq:kappastar}.
Note also that $\lambda_{\star} \in \left( 1/2, 1 \right)$, due to the fact that $a > b$.

We group the cases of~\eqref{eq:cond_exp_small_approx} and~\eqref{eq:cond_exp_med_approx} together, but treat them separately from the case of~\eqref{eq:cond_exp_big_approx}, which leads to defining the following random sets:
\begin{equation*}
 A_{t} := \left\{ i \in \left[ t \right] : R_{i-1} > \frac{a}{b} \right\},
 \qquad
 \qquad
 B_{t} := \left\{ i \in \left[ t \right] : R_{i-1} \leq \frac{a}{b} \right\}.
\end{equation*}
In words, the set $A_{t}$ is the set of time indices when the $\MAP$ estimator is equal to $1$ regardless of the private signal at this time.
Define also
\[
 R_{t}^{\left( 1 \right)} := \prod_{i \in A_{t}} \frac{R_{i}}{R_{i-1}},
 \qquad
 \qquad
 R_{t}^{\left( 2 \right)} := \prod_{i \in B_{t}} \frac{R_{i}}{R_{i-1}},
\]
and note that
$R_{t} = R_{t}^{\left( 1 \right)} R_{t}^{\left( 2 \right)}$,
since $\left\{ A_{t}, B_{t} \right\}$ is a partition of $\left[ t \right]$.
By~\eqref{eq:cond_exp_big_approx}, we have that 
there exists $C = C \left(  a, b \right)$ such that 
for any $\lambda \in \left[ 0, 1 \right]$ we have that 
\begin{equation}\label{eq:cond_exp_big_rewrite}
\E_{1} \left[ \left( \frac{R_{t}}{R_{t-1}} \right)^{-\lambda}  e^{f_{\lambda} p_{t}} \, \middle| \, R_{t-1}, t \in A_{t} \right]
\leq e^{C p_{t}^{2}}.
\end{equation}
Similarly, by~\eqref{eq:cond_exp_small_approx}, 
together with the fact that 
the right hand side of~\eqref{eq:cond_exp_small_approx} is greater than 
the right hand side of~\eqref{eq:cond_exp_med_approx} for all $t$ large enough, 
we have that there exists $C' = C' \left( a, b \right)$ such that 
for any $\lambda \in \left[ 0, 1 \right]$ we have that 
\begin{equation}\label{eq:cond_exp_small_rewrite}
\E_{1} \left[ \left( \frac{R_{t}}{R_{t-1}} \right)^{-\lambda}  e^{f_{1-\lambda} p_{t}} \, \middle| \, R_{t-1}, t \in B_{t} \right]
\leq e^{C' p_{t}^{2}}.
\end{equation}
Let $M_{t} := \sum_{i=1}^{t} p_{i}$,
and note that by the choice of $\left\{ p_{t} \right\}_{t=1}^{\infty}$,
together with~\eqref{eq:max_value},
we have that
\begin{equation}\label{eq:Mt_asymptotics}
M_{t} = \left( 1 + o \left( 1 \right) \right) \left( 1 + \eps \right) \frac{a+b}{b} \kappa_{\star} \left( a, b \right) \log t
= \left( 1 + o \left( 1 \right) \right) \frac{ \left( 1 + \eps \right)}{f_{\lambda_{\star}} \left( a, b \right)} \log t
\end{equation}
as $t \to \infty$.
Define also the random variable
$\Gamma_{t} := \sum_{i \in A_{t}} p_{i}$.
Putting together~\eqref{eq:cond_exp_big_rewrite} and~\eqref{eq:cond_exp_small_rewrite}, 
it follows by induction that 
there exists $C = C \left( a, b \right)$ such that 
for any $\lambda_{1}, \lambda_{2} \in \left[0,1\right]$ we have that 
\begin{equation}\label{eq:exp_joint_lambda12}
\E_{1} \left[ \left( R_{t}^{\left( 1 \right)} \right)^{-\lambda_{1}} e^{f_{\lambda_{1}} \Gamma_{t}} \left( R_{t}^{\left( 2 \right)} \right)^{-\lambda_{2}} e^{f_{1-\lambda_{2}} \left( M_{t} - \Gamma_{t} \right)} \right]
\leq \exp \left( C \sum_{i=1}^{t} p_{i}^{2} \right).
\end{equation}
Since $\sum_{i=1}^{\infty} p_{i}^{2} < \infty$,
the expectation in~\eqref{eq:exp_joint_lambda12} is bounded above by a constant independent of $t$.
That is, there exists $C_{0} = C_{0} \left( a, b \right) < \infty$ such that
\begin{equation}\label{eq:exp_joint_lambda12_C0}
\E_{1} \left[ \left( R_{t}^{\left( 1 \right)} \right)^{-\lambda_{1}} e^{f_{\lambda_{1}} \Gamma_{t}} \left( R_{t}^{\left( 2 \right)} \right)^{-\lambda_{2}} e^{f_{1-\lambda_{2}} \left( M_{t} - \Gamma_{t} \right)} \right]
\leq C_{0}.
\end{equation}

Next, we claim that $R_{t} \leq \tfrac{a}{b}$ implies that
\begin{equation}\label{eq:Rt1leq1}
R_{t}^{\left( 1 \right)} \leq 1,
\end{equation}
and furthermore that there exists a constant $C_{1} < \infty$ such that
\begin{equation}\label{eq:Rt1_cantbetoosmall}
 \p_{1} \left( R_{t}^{\left( 1 \right)} \leq t^{-C_{1}} \right) \leq t^{-2}.
\end{equation}
We defer the proofs of both of these claims to Appendix~\ref{sec:computations_proofs}.

From~\eqref{eq:Rt1leq1} we get the following bound on the probability of interest:
\begin{equation}\label{eq:ub_two_cases}
\p_{1} \left( R_{t} \leq \frac{a}{b} \right)
\leq
\p_{1} \left( R_{t} \leq \frac{a}{b}, \ t^{-C_{1}} \leq R_{t}^{\left( 1 \right)} \leq 1 \right)
+
\p_{1} \left( R_{t}^{\left( 1 \right)} \leq t^{-C_{1}} \right).
\end{equation}
By~\eqref{eq:Rt1_cantbetoosmall} the second term is at most $t^{-2}$,
so in order to show~\eqref{eq:ub_goal} it suffices to bound from above the first term in the display above. 
We can break the event
$\left\{ R_{t} \leq \frac{a}{b}, \ t^{-C_{1}} \leq R_{t}^{\left( 1 \right)} \leq 1 \right\}$
into subevents based on the value of $R_{t}^{\left( 1 \right)}$.
Recall that
$R_{t} = R_{t}^{\left( 1 \right)} R_{t}^{\left( 2 \right)}$,
so if
$R_{t}^{\left( 1 \right)} \in \left[ e^{-\left( x + 1 \right)} , e^{-x} \right]$
and
$R_{t} \leq \tfrac{a}{b}$,
then
$R_{t}^{\left( 2 \right)} \leq \tfrac{a}{b} e^{x+1}$.
Letting
$C_{3} := 1 + \log \tfrac{a}{b}$
we obtain the bound
\begin{align}
\p_{1} \left( R_{t} \leq \frac{a}{b}, \ t^{-C_{1}} \leq R_{t}^{\left( 1 \right)} \leq 1 \right)
&\leq \sum_{x=0}^{C_{1} \log t} \p_{1} \left( R_{t}^{\left( 1 \right)} \in \left[ e^{- \left( x + 1 \right)} , e^{-x} \right], \ R_{t}^{\left( 2 \right)} \leq e^{x+C_{3}} \right) \notag \\
&\leq \sum_{x=0}^{C_{1} \log t} \p_{1} \left( R_{t}^{\left( 1 \right)} \leq e^{-x} , \ R_{t}^{\left( 2 \right)} \leq e^{x+C_{3}} \right). \label{eq:sum_x_bound}
\end{align}
We estimate each term in this sum. First, we can rewrite this probability as follows:
\begin{multline*}
\p_{1} \left( R_{t}^{\left( 1 \right)} \leq e^{-x} , \ R_{t}^{\left( 2 \right)} \leq e^{x+C_{3}} \right)
= \p_{1} \left( \left( R_{t}^{\left( 1 \right)} \right)^{-\lambda_{\star}} \geq e^{\lambda_{\star} x}, \ \left( R_{t}^{\left( 2 \right)} \right)^{- \left( 1 - \lambda_{\star} \right)} \geq e^{ - \left( 1 - \lambda_{\star} \right) \left( x + C_{3} \right)} \right) \\
= \p_{1} \left( \left( R_{t}^{\left( 1 \right)} \right)^{-\lambda_{\star}} e^{f_{\lambda_{\star}} \Gamma_{t}} \geq e^{\lambda_{\star} x + f_{\lambda_{\star}} \Gamma_{t}}, \ \left( R_{t}^{\left( 2 \right)} \right)^{- \left( 1 - \lambda_{\star} \right)} e^{f_{\lambda_{\star}} \left( M_{t} - \Gamma_{t} \right)} \geq e^{ - \left( 1 - \lambda_{\star} \right) \left( x + C_{3} \right) + f_{\lambda_{\star}} \left( M_{t} - \Gamma_{t} \right)} \right).
\end{multline*}
If both inequalities hold in the display above, then also the product of the expressions on the left hand sides is greater than or equal to the product of the expressions on the right hand sides. We thus obtain the following bound:
\begin{multline*}
\p_{1} \left( R_{t}^{\left( 1 \right)} \leq e^{-x} , \ R_{t}^{\left( 2 \right)} \leq e^{x+C_{3}} \right)  \\
\leq
\p_{1} \left( \left( R_{t}^{\left( 1 \right)} \right)^{-\lambda_{\star}} e^{f_{\lambda_{\star}} \Gamma_{t}}  \left( R_{t}^{\left( 2 \right)} \right)^{- \left( 1 - \lambda_{\star} \right)} e^{f_{\lambda_{\star}} \left( M_{t} - \Gamma_{t} \right)}   \geq  e^{ \left( 2 \lambda_{\star} - 1 \right) x + f_{\lambda_{\star}} M_{t} - \left( 1 - \lambda_{\star} \right) C_{3} }   \right).
\end{multline*}
Using Markov's inequality, together with~\eqref{eq:exp_joint_lambda12_C0} with $\lambda_{1} = \lambda_{\star}$ and $\lambda_{2} = 1 - \lambda_{\star}$, we obtain that 
\begin{align}
\p_{1} \left( R_{t}^{\left( 1 \right)} \leq e^{-x} , \ R_{t}^{\left( 2 \right)} \leq e^{x+C_{3}} \right)
&\leq
C_{0} e^{\left( 1 - \lambda_{\star} \right) C_{3}} \exp \left\{ - \left( 2 \lambda_{\star} - 1 \right) x - f_{\lambda_{\star}} M_{t} \right\} \notag \\
&\leq
C_{0} e^{\left( 1 - \lambda_{\star} \right) C_{3}} \exp \left\{ - f_{\lambda_{\star}} M_{t} \right\}, \label{eq:final_bound_of_term}
\end{align}
where the second inequality follows from the facts that $x \geq 0$ and $\lambda_{\star} > 1/2$.
Recalling from~\eqref{eq:Mt_asymptotics} that
$f_{\lambda_{\star}} M_{t} = \left( 1 + o \left( 1 \right) \right) \left( 1 + \eps \right) \log t$,
and using~\eqref{eq:sum_x_bound} and~\eqref{eq:final_bound_of_term}, we arrive at the following bound:
\[
\p_{1} \left( R_{t} \leq \frac{a}{b}, \ t^{-C_{1}} \leq R_{t}^{\left( 1 \right)} \leq 1 \right)
\leq
\frac{C_{4} \log t}{t^{\left( 1 + o \left( 1 \right) \right) \left( 1 + \eps \right)}}
\]
for some constant $C_{4} < \infty$.
Putting this together with~\eqref{eq:ub_two_cases} and~\eqref{eq:Rt1_cantbetoosmall} we obtain~\eqref{eq:ub_goal}.
\end{proof}

\subsection{A lower bound} \label{sec:lb_proof} 

\begin{proof}[Proof of Lemma~\ref{lem:map_estimator_two}~\eqref{lem:map_lb}]
By~\eqref{eq:map_error_bounds} our goal is to show that
\begin{equation}\label{eq:lb_goal}
\p_{1} \left( R_{t} < \frac{b}{a} \right) = \omega \left( \frac{1}{t} \right)
\end{equation}
as $t \to \infty$, and recall that we assume that the revealing probabilities $\left\{ p_{t} \right\}_{t=1}^{\infty}$ satisfy~\eqref{eq:pt_lim}.
As in the proof of the upper bound, let
$M_{t} := \sum_{i=1}^{t} p_{i}$
and note that~\eqref{eq:pt_lim} implies that
\begin{equation}\label{eq:Mt_ub}
M_{t}
\leq \left( 1 + o \left( 1 \right) \right) \left( 1 - \eps \right) \frac{a+b}{b} \kappa_{\star} \left( a , b \right) \log t
= \left( 1 + o \left( 1 \right) \right) \frac{1-\eps}{f_{\lambda_{\star}} \left( a, b \right)} \log t
\end{equation}
as $t \to \infty$,
where recall the definition of $f_{\lambda}$ and $\lambda_{\star}$ from Section~\ref{sec:ub_proof}. We may also assume that
\begin{equation}\label{eq:Mt_atleast_logt}
M_{t} \geq \delta \log t
\end{equation}
for all $t$ large enough, where
$\delta \equiv \delta \left( a, b \right) := \tfrac{1}{2 \log \left( 1 + a/b \right)}$;
if $M_{t} < \delta \log t$ then a simple argument shows that~\eqref{eq:goal_omega1t} holds, which we defer to Appendix~\ref{sec:small_Mt}.
Define also
$
\tau \left( s \right) := \min \left\{ t \geq 1 : M_{t} \geq s \right\},
$
and note that~\eqref{eq:Mt_ub} and~\eqref{eq:Mt_atleast_logt} imply that
$
\exp \left( \left( 1 + o \left( 1 \right) \right) \tfrac{f_{\lambda_{\star}}}{1-\eps} s \right)
\leq
\tau \left( s \right)
\leq
\exp \left( s / \delta \right)
$
as $s \to \infty$. In particular $\tau \left( s \right) < \infty$ for every $s < \infty$.

In order to show~\eqref{eq:lb_goal}, we define three events that together imply that $R_{t} < b/a$ and show that the probability that they all occur, given $\theta = 1$, is $\omega \left( 1 / t \right)$ as $t \to \infty$.
First, let
$t_{0} := \tau \left( 2 \tfrac{a+b}{a-b} \log\left( a/b \right) + 2 \right)$
and define
\begin{equation}\label{eq:initial_event}
\mathcal{E}_{0} := \left\{ R_{t_{0}} \leq \left( \tfrac{b}{a} \right)^{4} \right\}.
\end{equation}
This initial event takes the $L$-likelihood ratio below $b/a$, and the events we now define ensure that it stays below $b/a$.
Let $J_{t} := \log R_{t}$ denote the $L$-log-likelihood ratio and define the stopping time
$T := \min \left\{ s \geq t_{0} : J_{s} \notin \left[ - \log t, 2 \log \tfrac{b}{a} \right] \right\}$.
Define now the events
\begin{align*}
\mathcal{E}_{1} &:= \left\{ J_{T} \leq - \log t \right\}, \\
\mathcal{E}_{2} &:= \left\{ \min_{s \in \left[ t \right]} J_{s} \geq - \log^{3/4} t \right\}.
\end{align*}
Observe that $\mathcal{E}_{0}$, $\mathcal{E}_{1}$, and $\mathcal{E}_{2}$ together imply that
$J_{s} \in \left[ - \log^{3/4} t, 2 \log \tfrac{b}{a} \right]$
for all $s \in \left[ t_{0}, t \right]$.
In particular, they imply that $R_{t} < b/a$ and so
\begin{equation}\label{eq:lb_events}
\p_{1} \left( R_{t} < \frac{b}{a} \right) \geq \p_{1} \left( \mathcal{E}_{0} \cap \mathcal{E}_{1} \cap \mathcal{E}_{2} \right).
\end{equation}
In what follows we show that the right hand side of the display above is $\omega \left( 1 / t \right)$ as $t \to \infty$.

We start with the initial event $\mathcal{E}_{0}$.
Note that the first two individuals follow their private signal, that is, $Z_{1} = X_{1}$ and $Z_{2} = X_{2}$, and
hence if $X_{1} = X_{2} = 2$, then $R_{2} = \left( b / a \right)^{2}$.
If $X_{i} = 2$ for all $i \in \left\{ 3, 4, \dots, t_{0} \right\}$
then also $Z_{i} = 2$ for all $i \in \left\{ 3, 4, \dots, t_{0} \right\}$,
regardless of whether the corresponding players are Bayesians or revealers.
Consequently, by~\eqref{eq:p1Zt1} and~\eqref{eq:p2Zt1}, in this event the $L$-likelihood ratio at time $t_{0}$ is equal to
\begin{equation}\label{eq:Rt0}
R_{t_{0}} = \left( \frac{b}{a} \right)^{2} \prod_{i=3}^{t_{0}} \frac{1 - \frac{a}{a+b} p_{i}}{1 - \frac{b}{a+b} p_{i}}.
\end{equation}
Now using
\[
\frac{1 - \frac{a}{a+b} p_{i}}{1 - \frac{b}{a+b} p_{i}}
= 1 - \frac{\frac{a-b}{a+b} p_{i}}{1 - \frac{b}{a+b} p_{i}}
\leq 1 - \frac{a-b}{a+b} p_{i}
\leq e^{ - \frac{a-b}{a+b} p_{i} },
\]
and also, by the definition of $t_{0}$, the fact that
\[
\sum_{i=3}^{t_{0}} p_{i} \geq M_{t_{0}} - 2 \geq 2 \frac{a+b}{a-b} \log \frac{a}{b},
\]
we obtain from~\eqref{eq:Rt0} that $R_{t_{0}} \leq \left( b / a \right)^{4}$.
Hence
\[
\p_{1} \left( \mathcal{E}_{0} \right) \geq \p_{1} \left( X_{i} = 2 \ \forall i \in \left[ t_{0} \right] \right) = \left( \frac{b}{a+b} \right)^{t_{0}}.
\]
Since $t_{0}$ is a constant, we have that $\p_{1} \left( \mathcal{E}_{0} \right)$ is strictly bounded away from zero. Thus by~\eqref{eq:lb_events} it suffices to show that
\begin{equation}\label{eq:events1and2}
\p_{1} \left( \mathcal{E}_{1} \cap \mathcal{E}_{2} \, \middle| \, \mathcal{E}_{0} \right)
= \omega \left( \frac{1}{t} \right)
\end{equation}
as $t \to \infty$.

We now turn to estimating the probability of $\mathcal{E}_{1}$ and $\mathcal{E}_{2}$, given $\mathcal{E}_{0}$.
Note that, given $\mathcal{E}_{0}$, the $L$-log-likelihood ratio performs a random walk from time $t_{0}$ until the stopping time $T$.
Specifically, for $s \geq t_{0}$ we can write
\[
J_{s \wedge T} = J_{t_{0}} + \sum_{i = t_{0} + 1}^{s \wedge T} \xi_{i},
\]
where the random variables $\left\{ \xi_{i} \right\}_{i > t_{0}}$ are independent (of each other and everything else) with the following distribution under $\p_{1}$:
\begin{align*}
\p_{1} \left( \xi_{i} = \log \frac{a}{b} \right) &= \frac{a}{a+b} p_{i}, \\
\p_{1} \left( \xi_{i} = \log \frac{1 - \frac{a}{a+b} p_{i}}{1 - \frac{b}{a+b} p_{i}} \right) &= 1 - \frac{a}{a+b} p_{i}.
\end{align*}
Note, in particular, that the random variables $\left\{ \xi_{i} \right\}_{i > t_{0}}$ are uniformly bounded: $\log \tfrac{b}{a} \leq \xi_{i} \leq \log \tfrac{a}{b}$ for all $i > t_{0}$.
Furthermore, we have that
\[
\E_{1} \left[ \xi_{i} \right] = \frac{1 + \frac{a}{b} \left( \log \frac{a}{b} - 1 \right)}{1 + \frac{a}{b}} p_{i} + \Theta \left( p_{i}^{2} \right),
\]
showing that the $L$-log-likelihood ratio has an upward drift in this regime.
We perform a change of measure to remove this drift:
define $\wt{\p}_{1}$ such that under $\wt{\p}_{1}$ the random variables $\left\{ \xi_{i} \right\}_{i > t_{0}}$ have expectation zero.
That is, define $\wt{\p}_{1}$ such that for all $i > t_{0}$ we have that
\begin{align*}
\wt{\p}_{1} \left( \xi_{i} = \log \frac{a}{b} \right) &= q_{i}, \\
\wt{\p}_{1} \left( \xi_{i} = \log \frac{1 - \frac{a}{a+b} p_{i}}{1 - \frac{b}{a+b} p_{i}} \right) &= 1 - q_{i},
\end{align*}
where $q_{i}$ is chosen such that $\wt{\E}_{1} \left[ \xi_{i} \right] = 0$. A short computation gives that
\[
q_{i} = \frac{\log \frac{1 - \frac{b}{a+b} p_{i}}{1 - \frac{a}{a+b} p_{i}}}{\log \left( \frac{a}{b} \cdot \frac{1 - \frac{b}{a+b} p_{i}}{1 - \frac{a}{a+b} p_{i}} \right)}.
\]
In the following we first estimate~$\wt{\p}_{1} \left( \mathcal{E}_{1} \cap \mathcal{E}_{2} \, \middle| \, \mathcal{E}_{0} \right)$
and then show~\eqref{eq:events1and2} by understanding the Radon-Nikodym derivative of $\p_{1} \left( \cdot \, \middle| \, \mathcal{E}_{0} \right)$ with respect to $\wt{\p}_{1} \left( \cdot \, \middle| \, \mathcal{E}_{0} \right)$.

Under $\wt{\p}_{1}$ we have that $\left\{ J_{s \wedge T} \right\}_{s \geq t_{0}}$ is a martingale.
By the optional stopping theorem we thus have that
\begin{equation}\label{eq:optional_stopping}
\wt{\E}_{1} \left[ J_{T} \, \middle| \, \mathcal{E}_{0} \right]
=
\wt{\E}_{1} \left[ J_{t_{0}} \right].
\end{equation}
By the definition of $T$ we have that
either $J_{T} > 2 \log \tfrac{b}{a}$, in which case $J_{T} \in ( 2 \log \tfrac{b}{a}, \log \tfrac{b}{a} ]$,
or $J_{T} < - \log t$, in which case $J_{T} \in [ - \log t - \log \tfrac{a}{b}, - \log t )$.
Hence
we have that
\begin{align}
\wt{\E}_{1} \left[ J_{T} \, \middle| \, \mathcal{E}_{0} \right]
&=
\wt{\E}_{1} \left[ J_{T} \mathbf{1}_{\left\{ J_{T} < - \log t \right\}} \, \middle| \, \mathcal{E}_{0} \right]
+
\wt{\E}_{1} \left[ J_{T} \mathbf{1}_{\left\{ J_{T} > 2 \log \left( b / a \right) \right\}} \, \middle| \, \mathcal{E}_{0} \right] \notag \\
&\geq
\wt{\p}_{1} \left( \mathcal{E}_{1} \, \middle| \, \mathcal{E}_{0} \right) \times \left( - \log t - \log \tfrac{a}{b} \right) + \left( 1 - \wt{\p}_{1} \left( \mathcal{E}_{1} \, \middle| \, \mathcal{E}_{0} \right) \right) \times 2 \log \tfrac{b}{a} \notag \\
&\geq
2 \log \tfrac{b}{a} + \wt{\p}_{1} \left( \mathcal{E}_{1} \, \middle| \, \mathcal{E}_{0} \right) \times \left( - \log t \right). \label{eq:exp_logt}
\end{align}
From the definition of $t_{0}$ we also have that
\begin{equation}\label{eq:exp_Jt0}
\wt{\E}_{1} \left[ J_{t_{0}} \right] \leq 4 \log \frac{b}{a}.
\end{equation}
Thus putting together~\eqref{eq:optional_stopping},~\eqref{eq:exp_logt}, and~\eqref{eq:exp_Jt0} we obtain that
\begin{equation}\label{eq:E1}
\wt{\p}_{1} \left( \mathcal{E}_{1} \, \middle| \, \mathcal{E}_{0} \right) \geq \frac{2 \log \frac{a}{b}}{\log t}.
\end{equation} 
Applying Theorem 1.6 in~\cite{freedman} we have that 
\begin{equation*}
\wt{\p}_{1} \left( \min_{s \in \left[ t \right]} J_{s} < - \log^{3/4} t \, \middle| \, \mathcal{E}_{0} \right) \leq \exp \left( - c \log^{1/2} t \right)
\end{equation*}
for some constant $c > 0$. Together with~\eqref{eq:E1} this shows that
\begin{equation}\label{eq:E1E2_ptilde}
\wt{\p}_{1} \left( \mathcal{E}_{1} \cap \mathcal{E}_{2} \, \middle| \, \mathcal{E}_{0} \right)
\geq
\frac{\log \frac{a}{b}}{\log t}
\end{equation}
for all $t$ large enough.

Since we want to show~\eqref{eq:events1and2}, what remains is to estimate the Radon-Nikodym derivative of
$\p_{1} \left( \cdot \, \middle| \, \mathcal{E}_{0} \right)$
with respect to
$\wt{\p}_{1} \left( \cdot \, \middle| \, \mathcal{E}_{0} \right)$
on the sigma-algebra $\mathcal{F}_{t}$ and on the event $\mathcal{E}_{1} \cap \mathcal{E}_{2}$.
From the definition of $\wt{\p}_{1}$ we can write this down explicitly: we have that
\begin{equation}\label{eq:RN_prod}
\left. \frac{\mathrm{d} \p_{1} \left( \cdot \, \middle| \, \mathcal{E}_{0} \right)}{\mathrm{d} \wt{\p}_{1} \left( \cdot \, \middle| \, \mathcal{E}_{0} \right)} \right|_{\mathcal{F}_{t}}
\mathbf{1}_{\left\{ \mathcal{E}_{1} \cap \mathcal{E}_{2} \right\}}
=
\prod_{i=t_{0}+1}^{t} \left\{ \frac{\tfrac{a}{a+b} p_{i}}{q_{i}} \mathbf{1}_{\left\{ \xi_{i} = \log \frac{a}{b} \right\}} + \frac{1 - \tfrac{a}{a+b} p_{i}}{1 - q_{i}} \mathbf{1}_{\left\{ \xi_{i} = \log \frac{1 - \frac{a}{a+b} p_{i}}{1 - \frac{b}{a+b} p_{i}} \right\}} \right\}
\mathbf{1}_{\left\{ \mathcal{E}_{1} \cap \mathcal{E}_{2} \right\}}.
\end{equation}
We estimate each factor in this product. Specifically, we claim that there exists a constant $C = C \left( a, b \right)$ such that for every $i \in \left\{ t_{0} + 1, \dots, t \right\}$ the following holds:
\begin{equation}\label{eq:RN_estimate_i}
\frac{\tfrac{a}{a+b} p_{i}}{q_{i}} \mathbf{1}_{\left\{ \xi_{i} = \log \frac{a}{b} \right\}} + \frac{1 - \tfrac{a}{a+b} p_{i}}{1 - q_{i}} \mathbf{1}_{\left\{ \xi_{i} = \log \frac{1 - \frac{a}{a+b} p_{i}}{1 - \frac{b}{a+b} p_{i}} \right\}}
\geq
\exp \left( \left( 1 - \lambda_{\star} \right) \xi_{i} - f_{\lambda_{\star}} p_{i} - C p_{i}^{2} \right).
\end{equation}
This inequality can be checked for both potential values of $\xi_{i}$ by expanding the expressions in $p_{i}$.
When $\xi_{i} = \log \frac{1 - \frac{a}{a+b} p_{i}}{1 - \frac{b}{a+b} p_{i}}$, both sides of~\eqref{eq:RN_estimate_i} are equal to
$1 + \frac{a - b - a \log \frac{a}{b}}{\left( a + b \right) \log \frac{a}{b}} p_{i} + \Theta \left( p_{i}^{2} \right)$,
so by choosing $C$ large enough, the quadratic term on the right hand side will be smaller than that on the left hand side,
and hence~\eqref{eq:RN_estimate_i} holds in this case if $C$ is large enough.
When $\xi_{i} = \log \frac{a}{b}$,
the left hand side of~\eqref{eq:RN_estimate_i} is equal to
$\frac{a \log \frac{a}{b}}{a - b} + \frac{a \log \frac{a}{b}}{a - b} \left( \frac{a - b}{\left( a + b \right) \log \frac{a}{b}} - \frac{1}{2} \right) p_{i} + \Theta \left( p_{i}^{2} \right)$,
while the right hand side of~\eqref{eq:RN_estimate_i} is equal to
$\frac{a \log \frac{a}{b}}{a - b} - \frac{a \log \frac{a}{b}}{a - b}  f_{\lambda_{\star}} p_{i} + \Theta \left( p_{i}^{2} \right)$.
Thus the constant terms are equal and it can be checked that the coefficient of the first order term is greater for the expression on the left hand side than for the expression on the right hand side.
Hence in this case~\eqref{eq:RN_estimate_i} holds for all $i$ large enough regardless of the value of $C$, and if $C$ is chosen large enough then it holds for all $i \in \left\{ t_{0} + 1, \dots, t \right\}$.

After justifying~\eqref{eq:RN_estimate_i}, we can now turn back to estimating the quantity in~\eqref{eq:RN_prod} by multiplying~\eqref{eq:RN_estimate_i} over all $i \in \left\{ t_{0} + 1, \dots, t \right\}$.
Using the fact that $J_{t} = J_{t_{0}} + \sum_{i=t_{0}+1}^{t} \xi_{i}$ on the event $\mathcal{E}_{1} \cap \mathcal{E}_{2}$, and recalling also that $M_{t} = \sum_{i=1}^{t} p_{i}$, we obtain that
\begin{align}
\left. \frac{\mathrm{d} \p_{1} \left( \cdot \, \middle| \, \mathcal{E}_{0} \right)}{\mathrm{d} \wt{\p}_{1} \left( \cdot \, \middle| \, \mathcal{E}_{0} \right)} \right|_{\mathcal{F}_{t}}
\mathbf{1}_{\left\{ \mathcal{E}_{1} \cap \mathcal{E}_{2} \right\}}
&\geq
\exp \left( \left( 1 - \lambda_{\star} \right) \left( J_{t} - J_{t_{0}} \right) - f_{\lambda_{\star}} \left( M_{t} - M_{t_{0}} \right) - C \sum_{i=t_{0}+1}^{t} p_{i}^{2} \right)
\mathbf{1}_{\left\{ \mathcal{E}_{1} \cap \mathcal{E}_{2} \right\}}  \notag \\
&\geq
\exp \left( \left( 1 - \lambda_{\star} \right) J_{t} - f_{\lambda_{\star}} M_{t} - C' \right)
\mathbf{1}_{\left\{ \mathcal{E}_{1} \cap \mathcal{E}_{2} \right\}}, \label{eq:RN_estimate}
\end{align}
where
$C' := C \sum_{i=1}^{\infty} p_{i}^{2}$,
and where in the second inequality we used that $J_{t_{0}} < 0$ (given $\mathcal{E}_{0}$) and that $\sum_{i=1}^{\infty} p_{i}^{2} < \infty$ due to the assumption~\eqref{eq:pt_lim}.
Now recall from~\eqref{eq:Mt_ub} that
$f_{\lambda_{\star}} M_{t} \leq \left( 1 + o \left( 1 \right) \right) \left( 1 - \eps \right) \log t$
as $t \to \infty$.
Recall also that on the event $\mathcal{E}_{2}$ we have that
$J_{t} \geq - \log^{3/4} t$.
Plugging these two estimates into the right hand side of~\eqref{eq:RN_estimate} we obtain that
\begin{equation*}
\left. \frac{\mathrm{d} \p_{1} \left( \cdot \, \middle| \, \mathcal{E}_{0} \right)}{\mathrm{d} \wt{\p}_{1} \left( \cdot \, \middle| \, \mathcal{E}_{0} \right)} \right|_{\mathcal{F}_{t}}
\mathbf{1}_{\left\{ \mathcal{E}_{1} \cap \mathcal{E}_{2} \right\}}
\geq
\exp \left( - \left( 1 - \lambda_{\star} \right) \log^{3/4} t - C' \right) \times t^{ - \left( 1 + o \left( 1 \right) \right) \left( 1 - \eps \right)}
\mathbf{1}_{\left\{ \mathcal{E}_{1} \cap \mathcal{E}_{2} \right\}}.
\end{equation*}
Consequently we obtain that
\begin{align*}
\p_{1} \left( \mathcal{E}_{1} \cap \mathcal{E}_{2} \, \middle| \, \mathcal{E}_{0} \right)
&=
\wt{\E}_{1} \left[ \left. \frac{\mathrm{d} \p_{1} \left( \cdot \, \middle| \, \mathcal{E}_{0} \right)}{\mathrm{d} \wt{\p}_{1} \left( \cdot \, \middle| \, \mathcal{E}_{0} \right)} \right|_{\mathcal{F}_{t}}
\mathbf{1}_{\left\{ \mathcal{E}_{1} \cap \mathcal{E}_{2} \right\}} \, \middle| \, \mathcal{E}_{0} \right] \\
&\geq
\exp \left( - \left( 1 - \lambda_{\star} \right) \log^{3/4} t - C' \right) \times t^{ - \left( 1 + o \left( 1 \right) \right) \left( 1 - \eps \right)} \times
\wt{\p}_{1} \left( \mathcal{E}_{1} \cap \mathcal{E}_{2} \, \middle| \, \mathcal{E}_{0} \right)\\
&\geq
\frac{\log \frac{a}{b}}{\log t} \times \exp \left( - \left( 1 - \lambda_{\star} \right) \log^{3/4} t - C' \right) \times t^{ - \left( 1 + o \left( 1 \right) \right) \left( 1 - \eps \right)}
= \omega \left( \frac{1}{t} \right),
\end{align*}
where in the second inequality we used~\eqref{eq:E1E2_ptilde}.
This concludes the proof of~\eqref{eq:events1and2} as desired.
\end{proof}

%
%
%
%
%
%
%
%
%
%

\section{Conclusions and future work} \label{sec:future} 

In this paper we have shown that wrong information cascades can be \emph{broken} by a small number of individuals who disregard the actions of others and only follow their private signal,
leading to asymptotic learning of the correct action.
Moreover, we determined precisely the optimal asymptotic rate of decay of the error probability at a given time $t$.

This paper initiates a larger investigation into a broad family of problems that probe the fragility of information cascades.
We collect here some natural directions for future work.

\begin{itemize}
\item \textbf{Alternative objectives.} While asymptotic learning is the primary objective, once this is achieved one can ask to minimize various notions of error. In this paper we focused on the optimal asymptotic rate at which the error probability $\tE_{t}$ can go to zero, but other notions of error are natural to consider as well. For instance, what is the expected number of errors until time $t$ and how can this be minimized? That is, what is
\[
 \mathrm{NE}_{t} \equiv \mathrm{NE}_{t} \left( a, b \right) := \inf_{\left\{ p_{i} \right\}_{i=1}^{t}} \E \left[ \sum_{i=1}^{t} \mathbf{1}_{\left\{ Z_{i} \neq \theta \right\}} \right]
\]
and what are the optimal revealing probabilities to achieve this? Theorem~\ref{thm:main_two} implies that
$\mathrm{NE}_{t} \leq \left( 1 + o \left( 1 \right) \right) \kappa_{\star} \log t$.
We conjecture that $\mathrm{NE}_{t} = \left( 1 + o \left( 1 \right) \right) c \log t$ for some constant $c$ such that $0 < c < \kappa_{\star}$.

\item \textbf{Alternative behavioral models.} In this paper we considered a particular model: each player is either a Bayesian or follows their private signal blindly. More generally, one can consider any (causal) behavioral model that deviates from a pure Bayesian model. How do various notions of error depend on the model specifics?

\item \textbf{More general setups.} In this paper we considered the simplest possible setup: two possible states of the world, a uniform prior over them, with each state corresponding to a distribution over two possible private signals, which are in a natural bijection with the possible states of the world.
More generally one can ask the same questions with $k$ possible states of the world, a general prior over them, and each possible state of the world corresponding to a distribution over $\ell$ possible private signals.

\item \textbf{Unknown parameters.} We assumed that the players know the revealing probabilities $\left\{ p_{t} \right\}_{t=1}^{\infty}$. What if these are unknown, or if the players believe that they are $\left\{ p_{t} \right\}_{t=1}^{\infty}$ when they are actually $\left\{ q_{t} \right\}_{t=1}^{\infty}$? Does asymptotic learning occur? If so, what is the optimal learning rate?


\end{itemize}

We suspect that, for many of these problems, the results and techniques of this paper will be useful in determining the correct order of magnitude for the relevant quantities.
However, just like in this paper, determining the precise constants will require a deeper understanding of the specific problem of interest.


\section*{Acknowledgements}

We thank Louigi Addario-Berry for helpful comments. 


\bibliographystyle{abbrv}
\bibliography{bib}


\appendix

\section{Auxiliary proofs} \label{sec:computations_proofs} 

\subsection{Proof of~\eqref{eq:Rt1leq1}} \label{sec:Rt1leq1} 

To prove~\eqref{eq:Rt1leq1}, first note that the set $A_{t}$ is either empty, in which case $R_{t}^{\left( 1 \right)} = 1$,
or it can be decomposed into blocks of consecutive integers.
That is, there exist integers
$\ell_{1}, u_{1}, \ell_{2}, u_{2}, \dots, \ell_{k}, u_{k}$
such that
$\ell_{i} \leq u_{i}$
for every $i \in \left[ k \right]$,
$u_{i} + 1 < \ell_{i+1}$
for every $i \in \left[ k - 1 \right]$,
and
$A_{t} = \cup_{i=1}^{k} \left\{ \ell_{i}, \ell_{i+1}, \dots, u_{i} \right\}$.
We can then write
\begin{equation}\label{eq:Rt1}
R_{t}^{\left( 1 \right)} = \prod_{i=1}^{k} \prod_{j=\ell_{i}}^{u_{i}} \frac{R_{j}}{R_{j-1}}
= \prod_{i=1}^{k} \frac{R_{u_{i}}}{R_{\ell_{i} - 1}}.
\end{equation}
Since $\ell_{i} \in A_{t}$ we have that $R_{\ell_{i} - 1} > \tfrac{a}{b}$ for every $i \in \left[k\right]$,
and since $u_{i}+1 \notin A_{t}$ we have that $R_{u_{i}} \leq \tfrac{a}{b}$ for every $i \in \left[k\right]$;
the condition $R_{t} \leq \tfrac{a}{b}$ is necessary to ensure that this latter fact holds for $i=k$ as well, that is, that $R_{u_{k}} \leq \tfrac{a}{b}$.
Putting these together we obtain that each factor in~\eqref{eq:Rt1} is at most~$1$, showing that~\eqref{eq:Rt1leq1} holds.

\subsection{Proof of~\eqref{eq:Rt1_cantbetoosmall}} \label{sec:Rt1_cantbetoosmall} 

We first rewrite the probability in question by taking logarithms:
\[
\p_{1} \left( R_{t}^{\left( 1 \right)} \leq t^{-C_{1}} \right)
=
\p_{1} \left( \sum_{i \in A_{t}} \log \frac{R_{i}}{R_{i-1}} \leq - C_{1} \log t \right).
\]
Given $\theta = 1$ and $i \in A_{t}$,
the random variable $\log \tfrac{R_{i}}{R_{i-1}}$ has the following distribution:
it takes the value $\log \tfrac{b}{a}$ with probability $\tfrac{b}{a+b} p_{i}$
and it takes the value
$\log \tfrac{1 - \tfrac{b}{a+b} p_{i}}{1 - \tfrac{a}{a+b} p_{i}}$
with probability $1 - \tfrac{b}{a+b} p_{i}$.
Define $Y_{i} := \left( \log \tfrac{R_{i}}{R_{i-1}} \right) \wedge 0$.
Thus,
given $\theta = 1$ and $i \in A_{t}$,
the random variable $Y_{i}$
takes the value $\log \tfrac{b}{a}$ with probability $\tfrac{b}{a+b} p_{i}$
and it takes the value $0$ otherwise.
We then have that
\[
\p_{1} \left( \sum_{i \in A_{t}} \log \frac{R_{i}}{R_{i-1}} \leq - C_{1} \log t \right)
\leq
\p_{1} \left( \sum_{i \in A_{t}} Y_{i} \leq - C_{1} \log t \right)
=
\p_{1} \left( \sum_{i \in A_{t}} \frac{-Y_{i}}{\log \tfrac{a}{b}} \geq \frac{C_{1}}{\log \tfrac{a}{b}} \log t \right),
\]
and notice that $Y_{i}' := - Y_{i} / \log \tfrac{a}{b}$ is a Bernoulli random variable with expectation $\E Y_{i}' = \tfrac{b}{a+b} p_{i}$.
Let $\left\{ B_{i} \right\}_{i=1}^{\infty}$ be independent Bernoulli random variables with expectation $\E B_{i} = \tfrac{b}{a+b} p_{i}$.
We then have the following bound:
\[
\p_{1} \left( \sum_{i \in A_{t}} \frac{-Y_{i}}{\log \tfrac{a}{b}} \geq \frac{C_{1}}{\log \tfrac{a}{b}} \log t \right)
\leq
\p \left( \sum_{i=1}^{t} B_{i} \geq \frac{C_{1}}{\log \tfrac{a}{b}} \log t \right).
\]
Note that
$\E \sum_{i=1}^{t} B_{i} = \tfrac{a}{a+b} \sum_{i=1}^{t} p_{i} = \left( 1 + o \left( 1 \right) \right) C' \log t$
for some $C' = C' \left( a, b \right)$.
Hence by the multiplicative Chernoff bound we obtain that
\[
\p \left( \sum_{i=1}^{t} B_{i} \geq \frac{C_{1}}{\log \tfrac{a}{b}} \log t \right)
\leq \exp \left( - 2 \log t \right) = t^{-2}
\]
if $C_{1}$ is large enough, concluding the proof.

\subsection{Small revealing probabilities} \label{sec:small_Mt} 

Here we prove Lemma~\ref{lem:map_estimator_two}\eqref{lem:map_lb} in the case when
$M_{t} := \sum_{i=1}^{t} p_{i} < \delta \log t$,
where
$\delta \equiv \delta \left( a, b \right) := \tfrac{1}{2 \log \left( 1 + a/b \right)}$.
Note that the first two individuals follow their private signal, that is, $Z_{1} = X_{1}$ and $Z_{2} = X_{2}$.
Hence if $X_{1} = X_{2} = 2$, then $R_{2} = \left( b / a \right)^{2}$, and we are in the regime where the MAP estimator outputs $2$.
Thus if all the subsequent revealers have $2$ as their private signal then the MAP estimator continues to output $2$.
Let
$\Rev_{t} := \left\{ i \in \left[ t \right] : I_{i} = 1 \right\}$
denote the set of revealers until time $t$.
Recalling~\eqref{eq:map_error_theta1} we thus obtain the following lower bound:
\begin{equation}\label{eq:simple_lb}
\p \left( \MAP \left( Z_{1}, \dots, Z_{t-1}, X_{t} \right) \neq \theta \right)
\geq
\p_{1} \left( X_{1} = X_{2} = 2, \ X_{i} = 2 \ \forall i \in \Rev_{t} \right).
\end{equation}
Note that $\left| \Rev_{t} \right|$ is independent of $\theta$ and that the Chernoff bound gives that
$\p \left( \left| \Rev_{t} \right| > 1.8 \delta \log t \right) \leq 1/2$.
Thus we have that
\begin{align}
\p_{1} \left( X_{1} = X_{2} = 2, \ X_{i} = 2 \ \forall i \in \Rev_{t} \right)
&\geq
\frac{1}{2}
\p_{1} \left( X_{1} = X_{2} = 2, \ X_{i} = 2 \ \forall i \in \Rev_{t} \, \middle| \, \left| \Rev_{t} \right| \leq 1.8 \delta \log t \right) \notag \\
&\geq \frac{1}{2} \left( \frac{b}{a+b} \right)^{2 + 1.8 \delta \log t} = \frac{b^{2}}{2 \left( a + b \right)^{2}} \times t^{-0.9}. \label{eq:chernoff_final}
\end{align}
Putting together~\eqref{eq:simple_lb} and~\eqref{eq:chernoff_final} proves~\eqref{eq:goal_omega1t}.

\end{document}